\documentclass[12pt,a4paper,reqno]{amsart}
\usepackage{amsfonts}
\usepackage{amsthm}
\usepackage{amsmath}
\usepackage{amscd}
\usepackage{latexsym}
\usepackage{amssymb}
\usepackage[all]{xy}

\usepackage[dvips]{color}

\numberwithin{equation}{section}

\newtheorem{theorem}{Theorem}[section]

\newtheorem{corollary}[theorem]{Corollary}

\newtheorem{lemma}[theorem]{Lemma}

\newtheorem{proposition}[theorem]{Proposition}

\newtheorem{remark}[theorem]{Remark}

\newenvironment{notationnf}{
  \refstepcounter{theorem}
  \noindent\textbf{Notation \thetheorem.}
}

\def\RR{{\mathbb{R}}}
\def\NN{{\mathbb{N}}}

\def\CC{{\mathbb{C}}}

\newcommand{\Hb}{{\mathbf H}}

\newcommand{\Bc}{{\mathcal B}}
\newcommand{\Cc}{{\mathcal C}}

\newcommand{\Gc}{{\mathcal G}}
\newcommand{\Hc}{{\mathcal H}}

\newcommand{\Pc}{{\mathcal P}}

\newcommand{\Sc}{{\mathcal S}}

\newcommand{\Xc}{{\mathcal X}}

\newcommand{\SSF}{\mathfrak S}

\DeclareMathOperator{\re}{{\rm Re}\,}

\newcommand{\Ker}{\hbox{{\rm Ker}}\,}
\newcommand{\Ran}{\hbox{{\rm Ran}}\,}

\newcommand{\Tr}{\operatorname{Tr\,}}

\newcommand{\Ltwo}{L^2(\RR^3)}
\newcommand{\Hone}{\Hb^1(\RR^3)}

\newcommand{\CN}{{\Cc}_N}
\newcommand{\GN}{{\Gc}_N}
\newcommand{\PN}{{\Pc}_N}

\newcommand{\BH}{\mathcal{B}(\Hb^{1})}
\newcommand{\BL}{\mathcal{B}(L^2)}

\newcommand{\h}{\mathcal{H}}
\newcommand{\B}{\mathcal{B}(\h)}

\newcommand{\UC}{\mathcal{U}}

\newcommand{\uF}{\mathfrak{u}}

\newcommand{\GH}{Gl(\Hb^{1})}

\renewcommand{\a}{\alpha}
\newcommand{\D}{\Delta}
\newcommand{\F}{\Phi}

\renewcommand{\P}{\Psi}

\newcommand{\f}{\phi}

\newcommand{\noi}{\noindent}

\newcommand{\PI}[2]{\left\langle #1 , #2 \right\rangle}
\newcommand{\PIL}[2]{\left\langle #1 , #2 \right\rangle_{L^2}}
\newcommand{\PIH}[2]{\left\langle #1 , #2 \right\rangle_{\Hb^{1}}}

\begin{document}

\title[Stiefel and Grassmann manifolds in Quantum Chemistry]{Stiefel and Grassmann manifolds \\ in Quantum Chemistry}
\author{Eduardo Chiumiento}
\address[E. Chiumiento]{Departamento de de Matematica, FCE-UNLP \\
                                       Calles 50 y 115 \\
                                       (1900) La Plata \\
                                       Argentina                                       
                                           and     
                                       Instituto Argentino de Matem\'atica \\ `Alberto P. Calder\'on' \\ CONICET \\ Saavedra 15 3er. piso \\
(1083) Buenos Aires \\ Argentina.}
                                       
\email{eduardo@mate.unlp.edu.ar}  
\author{Michael Melgaard} 
\address[M. Melgaard]{School of Mathematical  Sciences \\
                                     Dublin Institute of Technology       \\
                                     Dublin 8, Republic of Ireland} 
\email{mmelgaard@dit.ie}

\thanks{The first author is partially supported by Instituto Argentino de Matem\'atica and CONICET}
 \thanks{The second author is supported by a Stokes Award (Science Foundation Ireland), grant 07/SK/M1208.}







\begin{abstract}
We establish geometric properties of Stiefel and Grassmann manifolds which arise in relation to Slater
type variational spaces in many-particle Hartree-Fock theory and beyond. In particular, we prove that
they are analytic homogeneous spaces and submanifolds of the space of bounded operators on the 
single-particle Hilbert space. As a by-product we obtain that they are complete Finsler manifolds. 
These geometric properties underpin state-of-the-art results on existence of solutions to Hartree-Fock 
type equations.
\end{abstract}

\maketitle

\section{Introduction}
\label{ecmm1:intro}
The \textit{Stiefel manifold in Quantum Chemistry} is defined by  
\begin{equation}
\Cc_{N}:= \left\{ \, (\phi_1 , \ldots , \, \phi_N) \in (\Hb^{1}(\RR^{3}))^N \, :  \,   \langle \phi_i, \phi_j \rangle_{L^{2}(\RR^{3})} =\delta_{ij}, \, 1 \leq   i, j \leq N \, \right\},  
\label{ecmm1:intro-CN}
\end{equation}
where $N \in \NN$ (typically the number of electrons) is fixed and $\Hb^1=\Hb^{1}(\RR^{3})$ is the Sobolev space of order one (the single-particle Hilbert space).  Let $\UC(\CC^N)$ be the unitary group of $n \times n$ matrices. The \textit{Grassmann manifold in Quantum Chemistry}, denoted by $\GN$, is defined to be the quotient  of the Stiefel manifold under the equivalence relation  
\[  (\phi_1 , \ldots , \, \phi_N) \sim (\psi_1 , \ldots , \, \psi_N) \,\text{ if }\, \sum_{i=1}^N U_{ij}\, \phi_i =\psi_j, \, j=1, \ldots, N,\, \text{ for some  $U \in \UC(\CC^N)$. }  \]
Motivated by state-of-the-art existence results on Hartree-Fock type equations, based on abstract critical point 
theory, the aim of this paper is to establish the fundamental geometric properties and structures of manifolds of this type by means of operator theoretical methods. 

More precisely, it turns out that $\CN$ may be regarded as the subset of the algebra of bounded operators $\mathcal{B}(\Hb^{1})$ consisting of  partial isometries with respect to the  $L^2(\RR^3)$ inner product that have a fixed $N$-dimensional initial space.  On the other  hand, $\GN$ may be identified with the set of rank $N$ projections in $\mathcal{B}(\Hb^{1})$ which are orthogonal with respect to the $L^2(\RR^3)$ inner product.
Our results include:
\begin{itemize}
\item $\Cc_{N}$ and $\GN$ are analytic homogeneous spaces of a Banach-Lie group $\UC$ -- for $\UC$, see (\ref{Banach-Lie group UC}) ; 
\item $\Cc_{N}$ and $\GN$ are analytic  submanifolds of $\Bc(\Hb^{1})$ ; 
\item $\Cc_{N}$ and $\GN$ are complete Finsler manifolds. 
\end{itemize}

In 1977 Lieb and Simon \cite{liebsimon77} proved existence of a ground state for the non-relativistic Hartree-Fock minimization problem by a variational approach. The set of admissible states in Hartree-Fock theory consists of the Slater determinants 
\begin{equation*}
\Sc_{N} = \left\{ \, \P_{\rm e} \in \Hc_{\rm e} \, : \, \exists \F = \{ \f_{n} \}_{1 \leq n \leq N}  \in \Cc_{N} \, , \P_{\rm e} = 
\frac{1}{\sqrt{N !} } {\rm det}\, \left( \f_{n}(x_{m}) \right) \right\}
\label{ecmm1:intro-slaterset}
\end{equation*}
where $\Cc_{N}$ is the Stiefel manifold in (\ref{ecmm1:intro-CN}) and $\Hc_{\rm e}= \bigwedge^{N} \Hb^{1}(\RR^{3};\CC^{2})$, 
i.e., the $N$-particle Hilbert space consisting of antisymmetric spinor-valued functions; $\Sc_{N}$ does not form a vector space.
The components of the minimizer satisfy the associated Hartree-Fock equations (i.e., the associated Euler-Lagrange equations).  
Prior to \cite{liebsimon77}, the Hartree-Fock equations were studied by more direct approaches, yielding less general results. 
The Hartree-Fock problem is hard because electrons may escape to infinity (ionization) which, mathematically, corresponds to a 
loss of compactness. Subsequently, Lions \cite{lions87} came up with a new approach which enabled him to prove existence of 
infinitely many solutions to the non-relativistic Hartree-Fock equations, including a minimizer. The afore-mentioned loss of 
compactness can be expressed by saying that the Hartree-Fock functional does not satisfy the Palais-Smale condition, as first 
noticed by Lions \cite{lions87}. He developed a new strategy based upon constructing ``approximate critical points" with some 
information on the Hessian at these points. Lions' method for recovering compactness from second order information was later 
pursued in its full generality by Fang and Ghoussoub \cite{fanggh94,gh93}, in particular leading to streamlined versions of Lions's work. 
The Lions-Fang-Ghoussoub approach have been implemented for various Hartree-Fock type variational problems. For the non-relativistic 
Hartree-Fock setting, it was applied by Fang and Ghoussoub in \cite{fanggh94}. For the quasi-relativistic Hartree-Fock problem, 
wherein one replaces the kinetic energy operator $-\D$ (the negative Laplacian) by its quasi-relativistic analogue, 
$\sqrt{ -\a^{-2} \D + \a^{-4}} -\a^{-2}$ ($\a$ being the fine-structure constant), it was implemented by Enstedt and Melgaard \cite{memm09}. 
In the presence of an external magnetic field, existence of infinitely many distinct solutions to the magnetic Hartree-Fock equations, 
including a minimizer associated with a ground state, was established by Enstedt and Melgaard \cite{memm10a} in a general 
framework which, in particular, includes the following three examples of external fields: a constant magnetic field, a decreasing 
magnetic field, and a ``physically measurable" magnetic field. Going beyond the standard Hartree-Fock problem, by replacing 
single Slater determinants by finite linear combinations of the former, Lewin implemented the approach in the non-relativistic 
multi-configurative case \cite{lewin04}; inspired by Esteban and S\'er\'e \cite{mjesere99} who provided the first rigorous results 
on the Dirac-Fock equations.

The abstract critical point theoretical result, which lies at the heart of the Fang-Ghoussoub approach to multiple solutions 
\cite[Theorem~1.7]{fanggh94}, requires that the underlying manifold, i.e. $\Cc_{N}$ in the non-relativistic setting, is a complete, $C^{2}$ Riemannian manifold. The streamlined approach to existence of a ground state applies the 
perturbed variational principle by Borwein-Preiss \cite[Theorem~2.6]{borpreiss87} which demands that $\Cc_{N}$ is a complete metric space. We shall verify these requirements but, in fact, we shall establish stronger results.

Despite the fundamental importance of these variational spaces, there seems to be few results on their geometry. On the level of theoretical physics, algebraic properties of variational spaces for electronic 
calculations are studied in \cite{cassam94} and, in particular, the finite-dimensional Grassmann manifold is 
discussed in \cite{roweryman80}. Other related papers, using techniques similar to the ones in the present work, 
lie in the area of differential geometry of operators, and have no direct relevance for Quantum Chemistry. For instance, see \cite{ACM05,chiumiento10} for  unitary orbits of partial isometries,  \cite{AL08,BTR07,corachpr97}  
for unitary orbits of projections  and the references given in each of these articles. For the various infinite dimensional manifolds found in the later works, ad hoc proofs are needed in each particular example. In our case, it is interesting to remark that we make use of standard facts of the theory of Hilbert spaces with two norms, which was independently developed by Krein \cite{krein47,krein98} and Lax \cite{lax54}.

This paper is organized as follows. In Section \ref{B-Lie group} we introduce  the  group $\UC$, which acts transitively on the Stiefel and Grassmann manifolds in Quantum Chemistry. Then  we show some basic facts on $\UC$, including  that $\UC$ is a Banach-Lie group endowed with the norm topology of $\BH$. Section \ref{Stiefel} contains the main results on the differential structure of the Stiefel manifold, namely that $\CN$ is a real analytic homogeneous space of $\UC$ and  a real analytic submanifold of $\BH$. 
In Section \ref{Grassmann} we prove the corresponding results on the differential structure of the Grassmann manifold, which now follows as a consequence of the  fact that $\GN$ is a quotient space of $\CN$. In Section \ref{Finsler-Riemann}, as an application of the preceding results,  we show that $\CN$ and $\GN$ are complete Finsler manifolds.

\section{A Banach-Lie group}
\label{B-Lie group}
Let $L^2=\Ltwo$ be the space of equivalence classes of complex-valued functions $\phi$ which are Lebesgue measurable and satisfy 
$\int_{\RR^3} | \phi |^2\, dx < \infty$. It is a complex and separable Hilbert space with the inner product $\PIL{\phi}{\psi}=\int_{\RR^3} \phi \, \overline{\psi} \, dx$. The induced norm is denoted by $\| \, \cdot \, \|_{L^2}$.  

Let $\Hb^{1}=\Hone$ be the Sobolev space given by
\[ \Hb^{1}=\{  \, \phi \in L^2  \, : \, \exists \, \partial_i \phi \in L^2, \, i=1, \ldots, 3  \, \}, \]
where  $\partial_i \phi$ denotes the weak partial derivative with respect to $x_i$.   It is a complex and separable Hilbert space endowed with the inner product $\PIH{\phi}{\psi}= \PIL{\phi}{\psi} + \sum_{i=1}^3 \PIL{\partial_i \phi}{ \partial_i \psi}$. The corresponding norm is denoted by $\| \, \cdot \, \|_{\Hb^{1}}$. 
One has that $\|\phi\|_{L^2}\leq \| \phi \|_{\Hb^{1}}$, $\phi \in \Hb^{1}$ and, moreover, $\Hb^1$ is dense in $L^2$.


\medskip

\begin{notationnf}Let $\BH$ (resp. $\BL$) denote the algebra of bounded linear operators on $\Hb^{1}$ (resp. $L^2$). The symbol $\| \, \cdot \, \|$ denotes the usual operator norm on $\Hb^{1}$, meanwhile $\| \, \cdot \, \|_{\BL}$ denote the usual operator norm on $L^2$.  
\end{notationnf}

\begin{remark}
We need some basics facts on two well-known Banach-Lie groups, which are the groups of invertible  and unitary operators on a Hilbert space (see \cite{beltita06}, \cite{upmeier85}).  Actually, we will be concerned with only two special examples:
\begin{enumerate}
	\item The group $Gl(\Hb^{1})$ of invertible operators on $\BH$  is  a Banach-Lie group with the topology defined by the operator norm. Its Lie algebra is equal to $\BH$ endowed with the bracket $[X,Y]=XY- YX$. Moreover, the exponential map is the usual exponential of operators.
\item The unitary group $\UC(L^2)$ of the Hilbert space $L^2$. It is a real Banach-Lie group in the norm topology and its Lie algebra is given by the skew-hermitian operators on $L^2$.   Again the exponential map is the usual exponential of operators.
\end{enumerate}
\end{remark}

We would like to find a Banach-Lie group that acts transitively on the Stiefel and Grassmann manifolds in Quantum Chemistry. This job seems to be done by the following group
\begin{equation}\label{Banach-Lie group UC}
\UC := \{  \,  U \in \GH \, : \, \|U\phi\|_{L^2}=\|\phi\|_{L^2},\, \forall \, \phi \in \Hb^{1} \, \}.  
\end{equation}

\noi Since we could not find references to this group in the literature,  we shall prove some basic facts on its differential structure.  The next lemma provides different characterizations of $\UC$. 

\begin{lemma}\label{equivalence UC}
The following conditions are equivalent:
\begin{enumerate}
\item[i)] $U \in \UC$.
\item[ii)] $U \in \GH$ and $\PIL{U\phi}{U\psi}=\PIL{\phi}{\psi}$ for all $\phi, \psi \in \Hb^{1}$. 
\item[iii)] There exists $W \in \mathcal{U}(L^2)$ such that $W(\Hb^{1})=\Hb^{1}$ and $W|_{\Hb^{1}}=U$.
\item[iv)] $U \in \GH$ and $\PIL{U \phi}{\psi}=\PIL{\phi}{U^{-1}\psi}$ for all $\phi, \psi \in \Hb^{1}$.
\end{enumerate}
\end{lemma}
\begin{proof}
$i) \Leftrightarrow ii)$ The proof is analogous to the characterization of unitary operators (or isometries) in  Hilbert spaces. Let $\phi, \psi \in \Hb^{1}$ and $c \in \CC$, then we have that $\| U (\phi + c\psi) \|_{L^2} = \| \phi + c \psi \|_{L^2}$. We can use the polar identity in $L^2$ to conclude that $\PIL{U\phi}{U\psi}=\PIL{\phi}{\psi}$. The converse is trivial.

\medskip

\noi $ii) \Rightarrow iii)$ Since $\|U \phi \|_{L^2} = \| \phi \|_{L^2}$ and $\Hb^{1}$ is dense in $L^2$, the operator $U$ extends uniquely to a bounded operator $W$ on $L^2$. Moreover, $W$ is unitary since it satisfies $\PIL{W\phi}{W\psi}=\PIL{\phi}{\psi}$.

\medskip

\noi $iii) \Rightarrow ii)$ First we show that $U \in \BH$. To see this, let $(\phi_n)_n$ be a sequence in $\Hb^{1}$ such that $\| \phi_n  \|_{\Hb^{1}} \rightarrow 0$ and $\| U \phi_n - \psi \|_{\Hb^{1}} \rightarrow 0$. Then, $\phi_n \rightarrow 0$  in the topology of $L^2$ and $U \phi_n= W \phi_n \rightarrow 0$ in the topology of $L^2$. Hence we have $\psi=0$, and $U \in \BH$ by the closed graph theorem. 

Now we claim that $U \in \GH$. Indeed, since $U$ is a restriction of $W$, it is apparent that $U$ is injective. By assumption we have that $U$ is surjective. Thus we can use the open mapping theorem to prove our claim.   

Finally, we notice that  $\PIL{U\phi}{U\psi}=\PIL{\phi}{\psi}$ is a consequence of $U$ being a restriction of the unitary operator $W$. 

\medskip

\noi $iii) \Rightarrow iv)$ Under the same assumption, we have just proved that $U \in \UC$. Then, we have $U^{-1} \in \UC$. In particular, there exists $V \in \mathcal{U}(L^2)$ such that $V(\Hb^{1})=\Hb^{1}$ and $V|_{\Hb^{1}}=U^{-1}$. Notice that $VW\phi=VU\phi=U^{-1}U\phi=\phi$, for all $\phi \in \Hb^{1}$. Then we get $VW=I$, and in a similar way, $WV=I$. Thus we obtain $V=W^{\ast}$, and it follows that $U^{-1}=W^{\ast}|_{\Hb^{1}}$. Let $\phi, \psi \in \Hb^{1}$, then
\[ \PIL{U\phi}{\psi}=\PIL{W\phi}{\psi}=\PIL{\phi}{W^{\ast}\psi}=\PIL{\phi}{U^{-1}\psi}. \]

\medskip

\noi $iv) \Rightarrow i)$ We choose $\psi=U\phi$, then   $\|U\phi \|_{L^2} ^2= \PIL{\phi}{U^{-1}U\phi}=\|\phi\|_{L^2} ^2$.   
\end{proof}

\begin{remark}\label{inclusion smooth}
We claim that $\UC$ is a closed subgroup of $\GH$. In fact, it is clear that $\UC$ is a subgroup of $\GH$. Suppose that $(U_n)_n$ is a sequence in $\UC$ satisfying $\| U_n - U\| \rightarrow 0$, where $U \in \GH$. For any $\phi \in \Hb^{1}$, we see that 
\[ \|(U_n - U) \phi \|_{L^2} \leq \|(U_n - U) \phi \|_{\Hb^{1}} \leq \| U_n - U \| \| \phi \|_{\Hb^{1}} \rightarrow 0.
\] 
Hence we obtain $\|U \phi \|_{L^2} =\lim \|U_n \phi\|_{L^2}= \| \phi \|_{L^2}$, and our claim is proved. 
 
 It follows from a well known result on Banach-Lie groups (see \cite[Corollary 7.8]{upmeier85}) that  there exist on $\UC$  a Hausdorff topology and  an analytic  structure compatible with this topology making $\UC$ a  real Banach-Lie group with Lie algebra
\[  
\mathfrak{u}:=\{ \, X \in \BH \, : \, e^{tX} \in \UC, \, \forall \, t \in \RR\, \}.    
\]
Moreover,  the inclusion map $\UC \hookrightarrow \GH$ is analytic, its derivative at the identity is the inclusion map 
$\mathfrak{u} \hookrightarrow \BH$ and $\exp_{\UC}(X)=e^{X}$ for all $X \in \mathfrak{u}$. 
\end{remark}

\noi The following result  is the infinitesimal counterpart of Lemma \ref{equivalence UC}. In particular, it is worth pointing out that $i\uF$\,, where $i$ is the complex number, is a well studied class of operators usually known as symmetrizable operators (see \cite{krein98, lax54}). 

\begin{lemma}\label{lie algebra UC}
The following assertions are equivalent:
\begin{enumerate}
\item[i)] $X \in \uF$.
\item[ii)] $X \in \BH$ and $\PIL{X\phi}{\psi}=-\PIL{\phi}{X\psi}$ for all $\phi, \psi \in \Hb^{1}$. 
\item[iii)] There exists $Z \in \mathcal{B}(L^2)$ such that $Z^{\ast}=-Z$, $Z(\Hb^{1})\subseteq \Hb^{1}$ and $Z|_{\Hb^{1}}=X$. 
\end{enumerate} 
\end{lemma}
\begin{proof}
$i) \Rightarrow ii)$ By our assumption  the curve $\gamma(t)=e^{tX}$, $t \in \RR$,  is contained in $\UC$. Using Lemma \ref{equivalence UC} we can rewrite this fact as $\PIL{e^{tX}\phi}{\psi}=\PIL{\phi}{e^{-tX}\psi}$, for any $\phi, \psi \in \Hb^{1}$. Taking the derivative of $\gamma$ at $t=0$ we find that $\PIL{X\phi}{\psi}=-\PIL{\phi}{X\psi}$.

\medskip

\noi $ii) \Rightarrow i)$ Suppose that $X \in \BH$ and  $\PIL{X\phi}{\psi}=-\PIL{\phi}{X\psi}$ for all $\phi, \psi \in \Hb^{1}$. It is easily seen that
$$ \PIL{(tX)^n\phi}{\psi}=\left\{
\begin{array}{cc}
-\PIL{\phi}{(tX)^n\psi}              & \, \, \, \, \, \, \, \, \,  \, \, \, \, \, \,\text{if }  n \text{ is odd} ,\\
\PIL{\phi}{(tX)^n\psi}  & \, \, \, \, \, \, \, \, \,   \, \, \, \, \, \,\text{if } n \text{  is even}.
\end{array}\right.
$$
Therefore 
$$\PIL{\bigg(I + tX + \frac{(tX)^2}{2} + \ldots + \frac{(tX)^n}{n!}\bigg)\phi}{\psi}=\PIL{\phi}{\bigg(I - tX + \frac{(tX)^2}{2} + \ldots + (-1)^n\frac{(tX)^n}{n!}\bigg)\psi}.$$
Letting $n \rightarrow \infty$, we have $\PIL{e^{tX}\phi}{\psi}=\PIL{\phi}{e^{-tX}\psi}$. By Lemma \ref{equivalence UC} we conclude that $e^{tX} \in \UC$ for all $t \in \RR$, so $X \in \uF$.

\medskip


\medskip

\noi $iii) \Rightarrow ii)$ We will use the closed graph theorem  to show that $X \in \BH$. Let $(\phi_n)_n$ be a sequence in $\Hb^{1}$ such that $\| \phi_n \|_{\Hb^{1}} \rightarrow 0$ and $\|X \phi_n - \psi \|_{\Hb^{1}} \rightarrow 0$. Then, we have that $\|\phi_n \|_{L^2} \rightarrow 0$ and so $X\phi_n=Z\phi_n \rightarrow 0=\psi$ in the $L^2$ topology. We thus get $X \in \BH$. 

To complete the proof notice that $\PIL{X\phi}{\psi}=\PIL{Z\phi}{\psi}=-\PIL{\phi}{Z\psi}=-\PIL{\phi}{X\psi}$ for all $\phi, \psi \in \Hb^{1}$. 

\medskip

\noi $ii) \Rightarrow iii)$ The crucial point is to prove that $X$ is bounded with respect to the $L^2$-norm, which can be deduced  from  \cite[Theorem I]{krein98}. In fact, the operator $iX$ is symmetric with respect to the $L^2$ inner product.   Thus the operator $X$ has a bounded extension $Z$  to all of $L^2$, and it is not difficult to check that $Z^*=-Z$.
\end{proof}

\noi  As we stated in Remark \ref{inclusion smooth},  $\UC$ is a Banach-Lie group endowed with a topology that in general is stronger than the one defined by the norm of $\BH$. Actually, we have that both topologies coincide in this group because $\UC$ is an algebraic subgroup of $Gl(\Hb^{1})$ in the sense of \cite{harriskaup77}.

\begin{theorem}\label{topology}
The group $\UC$ is an algebraic subgroup of $Gl(\Hb^{1})$. In particular, $\UC$ is a real Banach-Lie group endowed with the norm topology of $\BH$, and its Lie algebra is given by
\[  \uF= \{ \,  X \in \BH  \,  :  \,  \PIL{X\phi}{\psi}=-\PIL{\phi}{X\psi} ,  \, \forall \,\phi, \psi \in \Hb^{1}  \,  \}. \]
\end{theorem}
\begin{proof}
We first prove that $\UC$ is an algebraic subgroup of $\GH$ of degree $\leq 2$. To see this, we define a family of complex-valued polynomials on $\BH \times \BH$   by 
$$ P_{\phi} (\,(X,Y) \,)= \PIL{YX\phi}{\phi} - \|\phi\|_{L^2}^2, \, \, \, \, \, \, \, \phi \in \Hb^{1}.$$
Then, we have
\[  \UC=\{ \, U \in \GH \, : \, P_{\phi}(\, (U,U^{-1})\,)=0, \, \forall \, \phi \in \Hb^{1}\,  \}. \]
Thus the assertion concerning the topology of $\UC$ follows from \cite[Theorem 1]{harriskaup77}. Finally, the characterization of the Lie algebra has already been  proved in Lemma \ref{lie algebra UC} $ii)$.
\end{proof}

\begin{remark}
According to Lemma \ref{equivalence UC}, we  have a unique unitary extension $W$ to $L^2$ of each operator $U \in \UC$.  Therefore operators in $\UC$ are in bijective correspondence with  
\[   \tilde{\UC}:=\{ \, W \in  \mathcal{U}(L^2) \, : \,  W(\Hb^{1})=\Hb^{1}            \,    \}.  \] 
However, it turns out that $\tilde{\UC}$ is not  closed in $\mathcal{U}(L^2)$. 
Let $(\phi_n)_n$ be a sequence on $\Hb^{1}$ such that $\| \phi_n \|_{L^2}=1$ and $\| \phi_n - \phi \|_{L^2} \to 0$, for some $\phi \in L^2 \setminus \Hb^1$.  Now consider the rank one $L^2$-orthogonal projections $P_n:=\PIL{\, \cdot \,}{\phi_n}\phi_n$ and $P:=\PIL{\, \cdot \,}{\phi}\phi$, and set
\[   W_n:=e^{iP_n}=I + (e^i -1)P_n \,. \]
Note that we have $W_n \in \tilde{U}$. Using that $\| \phi_n - \phi \|_{L^2} \to 0$, it follows that $\|W_n - W\|_{\BL}\to 0$, where $W:=e^{iP}=I + (e^i -1)P$. But  $W (\Hb^{1}) \neq \Hb^1$, by our choice  of the function $\phi$. 
\end{remark}

\section{Stiefel manifold in Quantum Chemistry}
\label{Stiefel}

Let $N \in \NN$.  The \textit{Stiefel manifold in Quantum Chemistry} is defined by  
\begin{equation}\label{def CN}
\CN:= \{ \, (\phi_1 , \ldots , \, \phi_N) \in (\Hb^{1})^N \, :  \, \PIL{\phi_i}{\phi_j}=\delta_{ij}, \, 1 \leq   i, j \leq N \, \}.   
\end{equation}
We consider the subspace topology on $\CN \subset (\Hb^{1})^N$, which may be defined by 
$$d_{\CN}(\Phi, \Psi)= \bigg( \, \sum_{i=1}^N \| \phi_i - \psi_i \|_{\Hb^{1}} ^2\, \bigg)^{1/2},$$ 
where $\Phi=(\phi_1 , \ldots , \, \phi_N)$ and $\Psi=(\psi_1 , \ldots , \, \psi_N)$ belong to $C_N$.

\medskip 
 

\medskip

\noi We will identify $\CN$ with a subset of $\mathcal{B}(\Hb^{1})$. Let $S$ be a subspace of $\Hb^{1}$ such that $\dim S=N$.  
We define the following Stiefel type manifold:  
\[ St(S):= \{ \, V \in \mathcal{B}(\Hb^{1}) \, : \, \Ker(V)^{\perp_2}=S  , \, \| V \xi\|_{L^2} = \| \xi \|_{L^2}, \, \forall \, \xi \in S  \, \}, \]
where $\perp_{2}$ denotes  the orthogonal complement with respect to the inner product of $L^2$. 
We consider  $St(S)$ endowed with the usual operator topology inherited from $\mathcal{B}(\Hb^{1})$.

\medskip

\begin{remark}\label{form cn}
Let $\xi_1 , \ldots, \xi_N$ be a basis of $S$ such that $\PIL{\xi_i}{\xi_j}=\delta_{ij}$. We can rewrite 
\[ St(S)=\bigg\{ \, \sum_{i=1}^{N} \PIL{\, \cdot \,}{\xi_i} \phi_i \, : \, (\phi_1 , \, \ldots \, , \phi_N) \in \CN \, \bigg\}. \]
Indeed, any operator of the form $V_{\Phi}=\sum_{i=1}^{N} \PIL{\, . \,}{\xi_i} \phi_i$ satisfies $\Ker(V_{\Phi})=S^{\perp_{L^2}} \cap \Hb^{1}$ and is isometric on $S$. Conversely, each $V \in St(S)$ can be expressed in this form, where $\phi_i=V \xi_i$ for $i=1, \ldots, N$.
\end{remark}

\begin{lemma}
 $\CN$ and $St(S)$ are homomorphic.
\end{lemma}  
\begin{proof}
Let $\{ \, \xi_1, \ldots, \xi_N \, \}$  be a basis of $S$ satisfying $\PIL{\xi_i}{\xi_j}=\delta_{ij}$. For each element $\Phi=(\phi_1, \ldots , \phi_N) \in \CN$, we set   
$$
V_{\Phi}  \xi = \sum_{i=1}^N \PIL{\xi}{\xi_i}  \phi_i \,,
$$
for all $\xi \in \Hb^{1}$. Then we have $V_{\Phi} \in St(S)$, and the map $\CN \longrightarrow St(S), \, \, \, \, \, \Phi \mapsto V_{\Phi}$, is a bijection. 
Moreover, this   map is a homomorphism. In fact, note that 
\[  
\| \phi_i - \psi_i \|_{\Hb^{1}} = \| V_{\Phi} \xi_i - V_{\Psi} \xi_i \|_{\Hb^{1}} \leq \| V_{\Phi} - V_{\Psi} \| \, \| \xi_i \|_{\Hb^{1}}\,.   
\]
Then, we have
\[ 
d_{\CN}(\Phi, \Psi)   \leq \sqrt{N} \max_{1 \leq i \leq N}  \| \xi _i \|_{\Hb^{1}} \, \, \| V_{\Phi} - V_{\Psi}\| .
\] 
On the other hand, let $\xi \in \Hb^{1}$ such that $\| \xi \|_{\Hb^{1}}=1$. Then we get
\begin{align*}
 \|(V_{\Phi} - V_{\Psi})\xi\| & = \bigg\| \sum_{i=1}^N \PIL{\xi}{\phi_i} (\phi_i - \psi_i) \bigg\|_{\Hb^{1}} 
 \leq \sum_{i=1}^N \| \xi \|_{L^2} \| \xi_i \|_{L^2} \| \phi_i - \psi_i \|_{\Hb^{1}} \\
&  
 \leq \sum_{i=1}^N  \| \phi_i - \psi_i \|_{\Hb^{1}} \leq \sqrt{N} \, d_{\CN}(\Phi, \Psi),
\end{align*}  
and hence we obtain $\|V_{\Phi} - V_{\Psi}\| \leq \sqrt{N} \, d_{\CN}(\Phi, \Psi)$.
\end{proof}

\medskip

\begin{notationnf}Bear in mind the above identification, throughout the remainder of the paper we will only use the notation $\CN$ to indicate indistinctly the $N$-tuple or the operator presentation of the Stiefel manifold in Quantum Chemistry.  
\end{notationnf}


 
  
\begin{lemma}\label{transitive}
The  map
\[ \UC \times \CN \longrightarrow \CN, \, \, \, \, \, \, \, U\cdot V = UV, \]
is a  transitive action of the Banach-Lie group $\UC$ on $\CN$.
\end{lemma}
\begin{proof}
Let $S$ be an $N$-dimensional subspace of $\Hb^{1}$. Let  $V \in \BH$    such that  $\Ker(V)^{\perp_2}=S$ and  $\|V\xi \|_{L^2}=\| \xi \|_{L^2}$ for all $\xi \in S$. It follows that $\| UV \xi \|_{L^2}=\|V\xi \|_{L^2}=\|\xi\|_{L^2}$ for all $\xi \in S$ and $U \in \UC$, and also that $\Ker(V)^{\perp_2}=S$. This shows that $UV \in \CN$, whenever $U \in \UC$ and $V \in \CN$, so the action is well-defined.

Let $V_0, V_1 \in \CN$. 
We need to find an $U \in \UC$ such that $U V_0=V_1$. Let $\xi_1, \ldots , \xi_N$ be a basis of $S$ such that $\PIL{\xi_i}{\xi_j}=\delta_{ij}$. Then
$S_0={\rm span} \, \{  \, V_0\xi_1 , \ldots, V_0\xi_N , V_1\xi_1 , \ldots, V_1\xi_N  \, \} \subseteq \Hb^{1}$ has finite dimension, say $R$,  with  $N \leq R \leq 2N$. Then we can construct two orthonormal basis of $S_0$ with respect to the  inner product of $L^2$, namely
$V_0\xi_1 ,\ldots, V_0\xi_N, \alpha_{N +1}, \alpha_{N + 2} , \ldots, \alpha_{R}$ and $V_1\xi_1 , \ldots, V_1\xi_N, \beta_{N +1}, \beta_{N + 2} , \ldots, \beta_{R}$. 
 
Note that $\Hb^{1}=S_0 \oplus (S_0^{\perp_{2}} \cap \Hb^{1})$, where  the sum is direct and both subspaces are closed in $\Hb^{1}$. 
Therefore we can define the required operator by

$$ U \xi :=\left\{
\begin{array}{ccc}
\displaystyle{\sum_{i=1}^N} c_i V_1 \xi_i + \displaystyle{\sum_{i=N+1}^{R}} c_i \beta_i  \, \, \, \, \, \, & \text{ if } \, \, \,  \, \, \, \, \, \, \xi=\displaystyle{\sum_{i=1}^N} c_i V_0 \xi_i + \displaystyle{\sum_{i=N+1}^{R}} c_i \alpha_i \in S_0  ,\\

\\
\xi \, \, \, \, \, \, & \text{ if } \, \, \,   \, \, \, \, \, \,  \xi \in S_0^{\perp_{L^2}} \cap \Hb^{1}.
\end{array}\right.
$$
We point out that $U$ leaves invariant $S_0$ and $S_0^{\perp_{L^2}}\cap \Hb^{1}$. It is apparent that $U \in \GH$. Also notice that for any $\xi \in \Hb^{1}$ we can write $$\xi=\sum_{i=1}^N c_i V_0 \xi_i + \sum_{i=N+1}^{R} c_i \alpha_i + \xi_0,$$ for some $c_i \in \CC$ and $\xi_0 \in S_0^{\perp_{L^2}}\cap \Hb^{1}$. Then,
\[ \|U\xi\|_{L^2}^2=\bigg\|\displaystyle{\sum_{i=1}^N} c_i V_1 \xi_i + \displaystyle{\sum_{i=N+1}^{R}} c_i \beta_i\bigg\|_{L^2}^2 + \| \xi_0\|_{L^2}^2=\sum_{i=1}^R |c_i|^2 + \|\xi_0\|_{L^2}^2=\|\xi\|_{L^2}^2.   \]
We thus get $U \in \UC$. Moreover, it is clear that $UV_0=V_1$.  
\end{proof}

\medskip

\subsection{Construction of continuous local cross sections}

Let $V \in \CN$. In this section we prove that the map $\pi_V: \UC \longrightarrow \CN$, $\pi_V(U)=UV$ has local continuous cross sections. We will use this result in the next section to show that $\CN$ is a real analytic homogeneous space of $\UC$ and a submanifold of $\BH$. 

 We begin by establishing the continuity of several maps.

\begin{remark}\label{square root}
Let $P,Q \in \BH$ be two projections  of rank $N$ such that  $P=P^{\ast}$ and $Q=Q^{\ast}$, where the adjoint is taken with respect to the $L^2$ inner product. Since the rank is finite, we may view $P$ and $Q$ as (continuous) orthogonal projections on $L^2$. We set
$A := (I-P)(I-Q)(I-P)$. Clearly it satisfies $0\leq A \leq I$, when one considers the order given by the cone of positive operators in $\BL$.
An easy computation shows that 
\[ A=I -P -Q +QP + PQ -PQP.  \]
Therefore $A=I + B$, where $B$ is a finite rank  $L^2$-self-adjoint operator. Moreover, note that $-I \leq B \leq 0$ and its range satisfies  
$\Ran(B)\subseteq \Hb^{1}$. We claim that the square root of $A$ satisfies  $A^{1/2}(\Hb^{1})\subseteq \Hb^{1}$. Here $A^{1/2}$ is 
defined as usual by the continuous functional calculus in $\BL$. To prove our claim we shall need a result on the convergence of the 
binomial series: the series
\[  (1+z)^{\alpha}=1 + \sum_{k=1}^{\infty}c_k z^k,  \, \, \, \, \, \, \, c_k=\binom{\alpha}{k}= \frac{\alpha(\alpha-1)(\alpha -2)\ldots(\alpha - k +1)}{k!},
\]
converges absolutely  for $|z|\leq 1$ whenever $\re (\alpha)>0$ (see for instance \cite[Theorem 247]{knopp51}). In particular, for $\alpha=1/2$ the power series converges uniformly on $|z|\leq 1$. Then we can define the square root of $A$ using the series, i.e.
\[ 
A^{1/2}=(I+B)^{1/2}=I + \sum_{k=1}^{\infty}c_k B^k. 
\]    
This operator series is convergent in the norm of $\BL$. Moreover, this definition coincides with the continuous functional calculus since 
the power series is uniformly convergent on $|z|\leq 1$ and the spectrum $\sigma(B|L^2)$ of $B$ on $L^2$ is contained in $[-1,0]$.

Now we can prove that $A^{1/2}(\Hb^1)\subseteq \Hb^1$.   Let $\xi \in \Hb^{1}$, then
\[   A^{1/2}\xi = \xi + B \bigg(\sum_{k=1}^{\infty} c_k B^{k-1}\bigg)\xi. 
\] 
As a consequence of the fact that $\Ran(B)\subseteq \Hb^{1}$, we can conclude $A^{1/2}\xi \in \Hb^{1}$, and our claim is proved.   
\end{remark}

\begin{lemma}\label{sqrt continuous}
Let $V \in \CN$. Then the map
\[ F:\CN \subseteq \BH \longrightarrow \BH , \, \, \, \, \,  F(W)=(\, (I- VV^{\ast})(I-WW^{\ast})(I-VV^{\ast})  \,)^{1/2} |_{\Hb^{1}} \]
is continuous, when the adjoint and the square root are with respect to the $L^2$ inner product. 
\end{lemma}
\begin{proof}
We first notice that according to Remark \ref{square root} the domain of $F(W)$ can be restricted to $\Hb^{1}$ in order to obtain a bounded operator on $\Hb^{1}$.

Let $(V_n)_n$ be a sequence in $\CN$ such that $\| V_n - V_0\| \to 0$. We set $P=VV^{\ast}$ and $Q_n=V_nV_n^{\ast}$, for $n \geq 0$. As in the Remark \ref{square root} we use the notation 
$$ B_n = -P -Q_n + PQ_n + Q_nP -PQ_nP ,  $$
 where $(I-P)(I-Q_n)(I-P)= I + B_n$. Notice that $B_n$ has finite rank, it is $L^2$ self-adjoint and $\sigma(B_n | L^2)\subseteq [-1,0]$. In particular, it follows that $\| B_n \|_{\BL} \leq 1$. 
 
 On the other hand,  it is worth noting that for each $n \geq 0$ the series
$$   \sum_{k=1} ^{\infty} c_k B_n ^k= c_1B_n + B_n\bigg(  \sum_{k=2}^{\infty} c_kB_n ^{k-2}  \bigg)B_n    $$
converges in the norm of $\BH$. In fact,  the series converges in $\BL$, and since $B_n$ has finite rank, it is forced to  be convergent in $\BH$. Moreover, we must have that $F(V_n)=I + \sum_{k=1}^{\infty} c_k B_n^{k}$.

Beside these remarks, we can now focus on the proof of the continuity of $F$. Fix $\epsilon >0$. Let $s \in \NN$,  then
\begin{align}
 \| F(V_n) - F(V_0)  \| & =  \bigg\| \sum_{k=1}^{\infty} c_k(B_n ^k - B_0 ^k) \bigg\| \nonumber    \\
& \leq \sum_{k=1}^{s} |c_k| \, \| B_n ^k - B_0 ^k\| +  \bigg\| \sum_{k=s+1}^{\infty} c_k B_n ^k \bigg\| + \bigg\| \sum_{k=s+1}^{\infty} c_k B_0 ^k \bigg\| \label{ineq1}
\end{align}
By the definition of $B_n$ we see that the first term on the right tends to zero for any fixed $s$. As we remarked in the previous paragraph, the series $\sum_{k=1}^{\infty} c_k B_0 ^k$ is convergent in the norm of $\BH$, so the third term can be made as small as we need. What  left is to show that the second  term can be smaller than $\epsilon$ for $s$ large enough. This is the main point because the convergence in $\BH$ of these series  depend on each $n\geq 1$.  Since   the norm in $\BL$  of each $B_n$ is never greater  than the norm in $\BH$ (see \cite[Theorem I]{krein98}), we cannot ensure that each series be absolutely convergent in $\BH$.    

In order to bound the second term we need to use that $V_n , V \in \CN$. Let $S$ be an $N$-dimensional subspace of $\Hb^{1}$ as in Remark \ref{form cn}.  Let $\xi_1 , \ldots , \xi_N$ be a basis of $S$ such that $\PIL{\xi_i}{\xi_j}=\delta_{ij}$. 
Recall that $B_n=P(-I + Q_n - Q_nP) + Q_n (-I +P)$, then
\begin{equation}\label{bound1}
\bigg\| \sum_{k=s+1}^{\infty} c_k B_n ^k  \bigg\| \leq \bigg\| P(-I + Q_n - Q_nP) \sum_{k=s+1}^{\infty} c_k B_n ^{k-1}   \bigg\| + \bigg\|  Q_n (-I +P)  \sum_{k=s+1}^{\infty} c_k B_n ^{k-1}    \bigg\| 
\end{equation}
An easy computation shows that $P=\sum_{i=1}^N \PIL{\, . \,}{V\xi_i}V\xi_i$. We write $X_n=-I + Q_n - Q_nP$ for short. Notice that $\| X_n \|_{\BL}\leq 3$.  Then the first term in (\ref{bound1}) can bounded as follows: for $\xi \in \Hb^{1}$, $\| \xi\|_{\Hb^{1}}=1$, 
\begin{align*}
 \bigg\| PX_n \sum_{k=s+1}^{\infty} c_k B_n ^{k-1}  \xi  \bigg\|_{\Hb^{1}} &
  =  \bigg\|  \sum_{i=1}^N \PIL{X_n \sum_{k=s+1}^{\infty} c_k B_n ^{k-1}\xi}{V\xi_i}V\xi_i    \bigg\|_{\Hb^{1}} \\
& \leq   \sum_{i=1}^N \bigg|\PIL{X_n \sum_{k=s+1}^{\infty} c_k B_n ^{k-1}\xi}{V\xi_i}  \bigg|    \,  \| V\xi \|_{\Hb^{1}} \\
& \leq \sum_{i=1}^N \|X_n\|_{\BL} \bigg\|   \sum_{k=s+1}^{\infty} c_k B_n ^{k-1}\xi  \bigg\|_{L^2} \| V\xi_i\|_{L^2} \,  \| V\| \\
& \leq 3 N \bigg(  \sum_{k=s+1}^{\infty} |c_k| \, \|B_n\|_{\BL}^{k-1} \bigg)  \| V\|  \leq 3 N\| V \|  \sum_{k=s+1}^{\infty} |c_k|.
\end{align*}
Thus we have
\begin{equation}\label{bound2}
\bigg\| P(-I + Q_n - Q_nP) \sum_{k=s+1}^{\infty} c_k B_n ^{k-1}   \bigg\| \leq 3 N \| V \|  \sum_{k=s+1}^{\infty} |c_k|.
\end{equation}
The second term on the right in (\ref{bound1}) can be bounded in a similar fashion, namely
\begin{align*}
\bigg\|  Q_n (-I +P)  \sum_{k=s+1}^{\infty} c_k B_n ^{k-1}  \xi  \bigg\| & 
\leq \sum_{i=1}^N \bigg| \PIL{ (-I +P)\sum_{k=s+1}^{\infty} c_k B_n ^{k-1} \xi}{V_n \xi_i}  \bigg| \, \| V_n \xi \|_{\Hb^{1}} \\
& \leq \sum_{i=1}^N   \bigg \| \sum_{k=s+1}^{\infty} c_k B_n ^{k-1} \bigg \|_{\BL} \| V_n \| \\
& \leq N \| V_n \| \sum_{k=s+1}^{\infty} |c_k |. 
\end{align*}
Since $(V_n)_n$ is convergent, then there exists $K>0$ such that $\| V_n \| \leq K$, for all $n \geq 1$. 
Therefore we have
\begin{equation}\label{bound3}
\bigg\|  Q_n (-I +P)  \sum_{k=s+1}^{\infty} c_k B_n ^{k-1} \bigg\| \leq N K \sum_{k=s+1}^{\infty} |c_k |. 
\end{equation}
Inserting (\ref{bound2}) and (\ref{bound3}) into (\ref{bound1}) we get
\[ \bigg\| \sum_{k=s+1}^{\infty} c_k B_n ^k  \bigg\| \leq N (3  \| V \| + K)  \sum_{k=s+1}^{\infty} |c_k|. \]
Finally, let $s$ be large enough to guarantee that 
\begin{enumerate}
\item[1.] $\displaystyle{\bigg\| \sum_{k=s+1}^{\infty} c_k B_0 ^k \bigg\| \leq \epsilon}$,
\item[2.] $\displaystyle{\sum_{k=s+1}^{\infty} |c_k| \leq \frac{\epsilon}{N (3  \| V \| + K)}}$.
\end{enumerate}
Using the estimates in (\ref{ineq1}), we arrive at
\[  \lim_{n \to \infty} \| F(V_n) - F(V_0) \| \leq  \lim_{n \to \infty} \sum_{k=1}^{s} |c_k| \, \| B_n ^k - B_0 ^k\| + 2 \epsilon=2 \epsilon.   \]
Since $\epsilon>0$ is arbitrary, our lemma follows.
\end{proof}

\begin{lemma}\label{cont projection}
The map $\CN \subseteq \BH \to \BH$, $W \mapsto WW^{\ast}$, is continuous, where the adjoint is with respect to the $L^2$ inner product.  
\end{lemma}
\begin{proof}
Let $V_1, V_2 \in \CN$. Let $S$ be an $N$-dimensional subspace contained in $\Hb^1$  such that  $\| V_j \xi \|_{L^2}=\|\xi\|_{L^2}$, for $j=1,2$. Let $\xi_1 , \ldots , \xi_N$ be a basis of $S$ such that $\PIL{\xi_i}{\xi_j}=\delta_{ij}$. Set $C:=\max_{1 \leq i \leq N} \| \xi_i \|_{\Hb^{1}}$.  For any $\xi \in \Hb^{1}$,  
\begin{align*}
\| (V_1V_1 ^{\ast}  - V_2 V_2^{\ast} )\xi\|_{\Hb^{1}} & = \bigg\| \sum_{i=1}^N  \PIL{\xi}{V_1\xi_i}V_1\xi_i - \sum_{i=1}^N  \PIL{\xi}{V_1\xi_i}V_1\xi_i \bigg\|_{\Hb^{1}} \\
& \leq     \sum_{i=1}^N  |\PIL{\xi}{(V_1 - V_2)\xi_i} | \, \|V_1\xi_i\|_{\Hb^{1}} +  |\PIL{\xi}{V_2\xi_i}| \, \| (V_1 - V_2 )\xi_i\|_{\Hb^{1}} \\
& \leq NC (C\|V_1\| + 1) \| V_1 - V_2 \|,    
\end{align*}
so we have $\| V_1V_1 ^{\ast}  - V_2 V_2^{\ast} \|\leq NC (C\|V_1\| + 1) \| V_1 - V_2 \|$, and the  stated continuity now follows. 
\end{proof}

\noi     It is worth pointing out  that the following construction of a continuous local cross section is adapted from \cite{ACM05}. In this article, the authors gave a continuous local cross section for a transitive action on each connected component of the set of  partial isometries in a $C^{\ast}$-algebra. The continuity of the square root stated in Lemma \ref{sqrt continuous} provides the technical tool to adapt the construction to our setting.

\begin{proposition}\label{sections cn}
Let $V \in \CN$. Then the map
\[  \pi_{V} : \UC \longrightarrow \CN , \, \, \, \, \, \, \, \pi_{V}(U)=UV, \]
has continuous local cross sections. In particular, it is a locally trivial fibre bundle.  
\end{proposition}
\begin{proof}
Let $S$ be an $N$-dimensional subspace of  $\Hb^{1}$ such that $\| V \xi \|_{L^2}=\|\xi\|_{L^2}$.  Let $\xi_1 , \ldots , \xi_N$ be a basis of $S$ such that $\PIL{\xi_i}{\xi_j}=\delta_{ij}$. Our estimates will involve the  constant:
$$C:=\max_{1 \leq i \leq N} \| \xi_i \|_{\Hb^{1}}.$$
We first prove that $\pi_V$ has  continuous local cross section in a neighborhood of $V$. 
We set 
$$ r_V:=\min \bigg\{ \, 1,  \, \frac{1}{  C^2 N^2  (1+\| V \|) (1 +CN + CN\|V\| )^2} \, \bigg\}.$$ 
Then consider the following open set
\[  \mathcal{W}=\{ \, V_1 \in \CN  \,  :  \,   \| V_1 - V \| < r_V        \,  \}.   \] 
Let $VV^{\ast}=P$ and $V_1V_1 ^{\ast}=P_1$, where the adjoint is taken with respect the $L^2$ inner product. Then $P$ and $P_1$ are $L^2$-orthogonal projections of rank $N$.  If $V_1 \in \mathcal{W}$, we claim that 
\begin{equation*}
\| P - PP_1P \| < 1.
\end{equation*}
In fact, we have
\begin{align}
 \| P -PP_1P \| & \leq \| P \| \| P- P_1P \|   \leq \| V  \| \|V^{\ast}\|^2  \| V - V_1V_1^{\ast} V   \| \nonumber \\
 & = \| V \|  \| V^{\ast}\|^2 \| (I- V_1V_1^{\ast})(V- V_1 ) \|   \nonumber \\
 & \leq \|V\| \|V^{\ast}\|^2  \| I - V_1 V_1^{\ast} \| \| V - V_1 \|. \label{est a}
\end{align}  
The task is now  to  estimate each of these factors. Notice that for any $\xi \in \Hb^{1}$, $\| \xi \|_{\Hb^{1}}=1$, we obtain 
\begin{align*}
 \|V^{\ast} \xi\|_{\Hb^{1}} & =\bigg\| \sum_{i=1}^N \PIL{\xi}{V \xi_i} \xi_i \bigg\|_{\Hb^{1}}  \leq \sum_{i=1}^N | \PIL{\xi}{V\xi_i}| \,  \| \xi_i \|_{\Hb^{1}} \\
 & \leq \sum_{i=1}^N \| \xi \|_{L^2} \| V \xi_i \|_{L^2} \|\xi_i \|_{\Hb^{1}} \leq  \sum_{i=1}^N  \|\xi_i \|_{\Hb^{1}} \leq C N. 
\end{align*}
We thus get
\begin{align}\label{est no 1}
\| V^{\ast} \|  \leq CN.
\end{align}
The third factor in (\ref{est a}) can be bounded as follows 
\begin{align}
\| I- V_1 V_1^{\ast} \| & \leq 1 + \| V_1  \|  \| V_1^{\ast} \| \leq  1 + C N \| V_1 \|  \nonumber \\
& \leq 1 + CN( r_V + \|V\| ) \label{est no 2}  
 \leq 1 + CN + C N \|V\| ,  
\end{align}
where we use that $r_V \leq 1$. Inserting (\ref{est no 1}) and (\ref{est no 2}) into (\ref{est a}), our claim follows. Therefore the operator $PP_1P$ is invertible on $\Ran(P)$. Notice that  $PP_1P$ and its inverse $(PP_1P)^{-1}$ on $\Ran(P)$  are bounded in the $L^2$ norm due to the fact that $P$ has finite rank.   Thus the square root with respect to $L^2$ of the positive operator $(PP_1P)^{-1}$ is well defined and bounded on $L^2$.  
Set $T_1=P_1 (PP_1P)^{-1/2}$, and then notice that
\[ 
T_1^{\ast} T_1= (PP_1P)^{-1/2} P_1 (PP_1P)^{-1/2}=(PP_1P)^{-1/2}(PP_1P)(PP_1P)^{-1/2}=P, 
\]
where the adjoint of $T_1$ is considered with respect to $L^2$. Our next step is to prove that $T_1T_1^{\ast}=P_1$. To this end we check that $P_1P=T_1|P_1P|$ is actually the polar decomposition. By the uniqueness of this decomposition, it is enough to show that
\[  T_1|P_1P|=P_1 |P_1P|^{-1}|P_1P|=P_1P,  \]
and 
\begin{equation}\label{range}
\Ran(P)=\Ran(PP_1P) \subseteq \Ran(PP_1) \subseteq \Ran(P),
\end{equation}
so we have
$\Ker(T_1)=\Ker(P)=\Ran(P)^{\perp_2}=\Ran(PP_1)^{\perp_2}=\Ker(P_1P).$
Hence $T_1$ is the $L^2$-partial isometry given by the polar decomposition, and consequently, we get that $\Ran(T_1T_1^{\ast})=\Ran(P_1P)$. 

On the other hand, in the same manner as (\ref{est a})  we have 
\begin{align*}
\|P_1 - P_1PP_1   \| & \leq  \|V_1\| \|V_1^{\ast}\|^2  \| I - V V^{\ast} \| \| V - V_1 \|    \\
& \leq (r_V + \| V \|)C^2N^2 (1 + CN\|V\|) \, \| V - V_1\| \\
& \leq (1+\|V\|)C^2N^2(1 +CN\|V\|)\, \|V - V_1\| < 1.
\end{align*}
According to the preceding inequality, $P_1 - P_1PP_1$ is invertible on $\Ran (P_1)$. Then we can prove that 
$\Ran (P_1P)=\Ran(P_1)$ in the same way as in (\ref{range}) interchanging the roles of $P$ and $P_1$. From this later fact, we deduce that $T_1T_1^{\ast}=P_1$. 

In order to construct another partial isometry $T_2$ on $L^2$ such that $T_2^{\ast}T_2=I-P$ and  $T_2T_2^{\ast}=I-P_1$ we repeat the above argument with the projections $I-P$ and $I-P_1$. In fact, notice that
\begin{align*}
\| (I- P) -(I-P)(I-P_1)(I-P)\| & = \|(P_1 -P P_1)(I-P)\|  \\
& \leq \| V_1 ^{\ast} \| \| I-VV^{\ast}\|^2 \| V- V_1\|  <1  
\end{align*} 
In a similar fashion we can find that $\| (I- P_1) -(I-P_1)(I-P)(I-P_1)\| < 1.$
Then we have that $T_2=(I-P_1) (\,(I-P)(I-P_1)(I-P)\,)^{-1/2}$ is the required $L^2$-partial isometry implementing equivalence between $I-P$ and $I-P_1$.
Actually, the definition of $T_2$ needs the following remarks:
\begin{itemize}
	\item[1.] Note that the finite rank projections $P$, $P_1$ can be extended to $L^2$, then the operator $A:=(I-P)(I-P_1)(I-P)$ can also be extended to $L^2$. Apparently, $A$ is positive with respect to the $L^2$ inner product, then $A^{1/2}$ is well defined using the continuous functional calculus in $\BL$.
	\item[2.] Recall that $A$ is invertible on $(I-P)(\Hb^1)$. It follows from \cite[Theorem II]{krein98} that $\sigma(A|(I-P)(L^2))\subseteq \sigma(A|(I-P)(\Hb^{1}))$, and consequently, the extension of $A$ to $(I-P)(L^2)$ is invertible. Then $A^{1/2}((I-P)(L^2))= (I-P)(L^2)$, and  by Remark  \ref{square root} we have $A^{1/2}((I-P)(\Hb^{1}))=A^{1/2}(\Hb^{1})\subseteq  \Hb^{1}$, so we find that $A^{1/2}((I-P)(\Hb^{1}))\subseteq (I-P)(\Hb^{1})$.
	\item[3.] According to our last remark, it is possible to restrict the domain of $A^{1/2}$, and the resulting operator $A^{1/2}|_{(I-P)(\Hb^{1})}:(I-P)(\Hb^{1}) \longrightarrow (I-P)(\Hb^{1})$ is bounded with respect to the inner product of $\Hb^{1}$. Moreover,  $C:= A^{-1}|_{(I-P)(\Hb^{1})} A^{1/2} |_{(I-P)(\Hb^{1})}$, which is also continuous with respect the topology of $\Hb^{1}$, turns out to be the inverse of $A^{1/2}|_{(I-P)(\Hb^{1})}$. 
	\item[4.] Observe  that $ T_2((I-P)(\Hb^{1}))  \subseteq (I-P_1)(\Hb^{1})$ and $T_2^*((I-P_1)(\Hb^{1}))  \subseteq (I-P)(\Hb^{1})$. Since  $T_2$ is an $L^2$-partial isometry from $(I-P)(L^2)$ onto $(I-P_1)(L^2)$,  then we find that $T_2((I-P)(\Hb^{1}))=(I-P_1)(\Hb^{1})$.
\end{itemize}
 
We define $T:=T_1 + T_2$, which  is clearly an unitary operator in $\BL$. Then  we note  that $T(\,P(\Hb^{1})\,)=T_1(\,P(\Hb^{1})\,)= P_1(\Hb^{1})$ and $T(\,(I-P)(\Hb^{1})\,)=T_2(\,(I-P)(\Hb^{1})\,)= (I- P_1)(\Hb^{1})$, so we have that $T(\Hb^{1})=\Hb^{1}$.  By Lemma \ref{equivalence UC} we get $T|_{\Hb^{1}} \in \UC$.

On the other hand, we set $W=  V_1 V^{\ast} T^{\ast} + I -P_1$. Clearly, $W$ is a unitary on $L^2$ such that $W(\Hb^{1}) \subseteq \Hb^{1}$. Moreover, $W^{\ast}=T V V_1^{\ast} + I-P_1$ also satisfies $W^{\ast}(\Hb^{1}) \subseteq \Hb^{1}$. Therefore we have $W(\Hb^{1})=\Hb^{1}$,  and consequently,  $W|_{\Hb^{1}} \in \UC$.

Now we give the continuous local cross section of $\pi_V$, namely
$$  \sigma : \mathcal{W} \longrightarrow \UC, \, \, \, \, \, \, \sigma(V_1)=W|_{\Hb^{1}} \, T|_{\Hb^{1}}\,. $$
Note that for $V_1 \in \mathcal{W}$, 
\[
\sigma(V_1)V = V_1 V^{\ast} T^{\ast} T V + (I-P_1)TV = V_1 V^{\ast} P V = V_1 V^{\ast} V = V_1 V_1^{\ast}V_1 =V_1,
\]
which shows that $\sigma$ is a section for $\pi_V$. The continuity of $\sigma$  can be deduced from the following facts:
\begin{itemize}
\item[1.] According to Lemma \ref{cont projection}  the map $\CN \to \BH$, $V_1 \mapsto V_1V_1^{\ast}$, is continuous in $\BH$. 

\item[2.] The map $V_1 \mapsto (P V_1V_1^{\ast}P)^{1/2}$ is clearly continuous in $\BH$ since $P$ is a finite rank operator and the square root is continuous in $\BL$. Then the map given by $T_1(V_1)=V_1V_1^{\ast}(P V_1V_1^{\ast}P)^{-1/2}$ is continuous because taking inverses on $P(\Hb^{1})$ is continuous.

\item[3.] From Lemma \ref{sqrt continuous} we have that $A(V_1)= (\, (I- VV^{\ast})(I-V_1V_1^{\ast})(I-VV^{\ast})  \,)^{1/2} |_{\Hb^{1}}$ is continuous in $\BH$. By the third remark after the definition of $T_2$, we know that $A(V_1)$ is invertible on $(I-P)(\Hb^{1})$. Since taking inverses on $(I-P)(\Hb^{1})$ is continuous, we can conclude that
$T_2(V_1)=  (I-V_1V_1^{\ast})A(V_1)^{-1}$ is continuous with respect the norm of $\BH$.

\item[4.]  Now the continuity of $T(V_1)=T_1(V_1) + T_2(V_1)$ is a straightforward consequence of the previous facts. On the other hand, $W(V_1)=  V_1 V^{\ast} T(V_1)^{\ast} + I -V_1V_1^{\ast}$ is continuous in $\BH$ since $T(V_1) \in \UC$, which implies that
$T(V_1)^{\ast}|_{\Hb^{1}}=T(V_1)^{-1}$, and hence the desired continuity  can be deduced again of the continuity of taking inverses.

\end{itemize}
It only remains to show how one can construct a continuous section for $\pi_V$ in a neighborhood of $V_0 \in \CN$. Let $U \in \UC$ such that $UV=V_0$. Then the required section is given by
$\tilde{\sigma}: \tilde{\mathcal{W}} \longrightarrow \UC$, $\tilde{\sigma}(V_1)=U \sigma(U^{-1}V_1)$,
where 
\[  \tilde{\mathcal{W}}=\bigg\{ \, V_1 \in \CN  \,  :  \,   \| V_1 - V_0 \| < \frac{r_V}{\|U^{-1}\|}        \,  \bigg\}.  \] 
This finishes the proof, the detailed verification of this later fact is straightforward. 
\end{proof}

\medskip

\subsection{Differential structure of $\CN$}

The following result is a consequence of the implicit function theorem in Banach spaces, and it can be found in  
\cite[Proposition 1.5]{raeburn77}. 

\begin{lemma}\label{Raeburn}
Let $G$ be a Banach-Lie group acting smoothly on a Banach space $\Xc$. For a fixed $x_0 \in \Xc$, denote by $\pi_{x_0} : G \longrightarrow \Xc$ the smooth map $\pi_{x_0} (g) = g \cdot x_0$. Let $G\cdot x_0$ be the orbit of  $x_0$. Suppose that 
\begin{enumerate}
	\item[1.] $\pi_{x_0}$ is an open mapping, when regarded as a map from $G$ onto the orbit $G\cdot x_0$  (with the relative topology of $\Xc$).
	\item[2.] The differential $(d \pi_{x_0})_1 : (TG)_1 \longrightarrow \Xc$ splits: its kernel and range are closed complemented subspaces.
\end{enumerate}
Then $G\cdot x_0$ is a smooth submanifold of $\Xc$, and the map $\pi_{x_0}: G \longrightarrow G\cdot x_0$ is a smooth submersion. 
\end{lemma}

\noi  Let $V \in \CN$. The isotropy group at $V$ of the above defined action is given by
\[ G_V=\{  \,   U \in \UC \, : \, UV=V \, \}. \]
Suppose that $\xi_1, \ldots , \xi_N$ is the orthonormal basis with respect to the inner product of $L^2$ of $S=\Ker(V)^{\perp_2}$.  Consider the projection $P=VV^{\ast}$, which is  given by 
\[ P: \Hb^{1} \longrightarrow \Hb^{1}, \, \, \, \, \, \, \, P(\xi)=\sum_{i=1}^N \PIL{\xi}{V\xi_i} V\xi_i\,. \]
Then we can rewrite the isotropy group as
\begin{equation}\label{isotropy}
 G_V=\{  \,   U \in \UC \, : \, UP=P \, \}.
\end{equation}
Our main result on the differential structure of $\CN$ now follows.

\begin{theorem}\label{subman}
Let $V \in \CN$. Then the map 
\[ \pi_V : \UC \longrightarrow \CN, \, \, \, \, \, \, \pi_V(U)=UV\]
is a real analytic submersion, and induces on $\CN$ a homogeneous structure. Furthermore, $\CN$ is a real analytic submanifold of $\BH$. 
\end{theorem}
\begin{proof}
The proof consists in applying Lemma \ref{Raeburn} with $X=\BH$, $G=\UC$ and $x_0=V$. We first note that the action  $\UC \times \BH \to \BH$, $(U,X)\mapsto UX$ is an analytic map. Indeed, according to Remark \ref{inclusion smooth} we know that the inclusion map $i:\UC \hookrightarrow Gl(\Hb^{1})$ is analytic. Thus the action, which is given by multiplication in $\BH$, has to be analytic.

On the other hand, it follows at once from Proposition \ref{sections cn} that  $\pi_V$ is an  open map. The differential of $\pi_V$ at the identity is given by
\[    \delta_V:=(d\pi_V)_I : \uF \longrightarrow \BH , \, \, \, \, \, \delta( X)=XV,    \]  
The Lie algebra of the isotropy group computed in (\ref{isotropy}), which is the kernel of $\delta_V$, is given by 
\[ \mathfrak{g}_V=\{ \, X \in \uF \, : \, XP=0 \, \}.   \]   
It is clear that $\mathfrak{g}_V$ is closed in $\uF$. From Lemma \ref{lie algebra UC} $iii)$ there exists $Z \in \BL$, $Z^{\ast}=-Z$, such that $Z|_{\Hb^{1}}=X$. Then we have that $ZP=0$ implies $-PZ=(ZP)^{\ast}=0$, and consequently, $PX=0$. Thus we may represent $X$ as a $2 \times 2$ matrix with respect to the decomposition induced by $P$, that is
$$ 
X=\left(\begin{array}{cc}{0}&{0}\\{0}&{X_{22}}\end{array}\right).
$$
Therefore  $\mathfrak{g}_V$ is complemented in $\uF$. In fact, the subspace 
\[ 
\mathfrak{h}_V=\{  \,  X \in \uF \,  : \,  (I-P)X(I-P)=0 \, \}
\]
is a closed supplement of $\mathfrak{g}_V$ in $\uF$. 

It remains to prove that the range of $\delta_V$ is a closed complemented subspace of $\BH$. To this end we define the following bounded linear map:
\[  K: \BH \longrightarrow \BH, \, \, \, \, \, K(Y)=PYV^{\ast} + (I-P)YV^{\ast}. \]
Then a straightforward computation shows that  $\delta_V \circ K \circ \delta_V = \delta_V$. We thus get that $E:=\delta_V \circ K$ is a continuous idempotent onto the range of $\delta_V$, and the proof is complete.
\end{proof}

 Let $\h$ be a complex separable Hilbert space. Let $\B$ denote the algebra of bounded linear operators acting on $\h$.  Denote by 
 $\| \, \cdot \, \|$ the usual operator norm. By a \textit{symmetrically-normed ideal}  we mean a two-sided ideal $\SSF$ of $\B$  
 equipped with a norm $\| \, \cdot \, \|_{\SSF}$ satisfying
\begin{enumerate}
	\item[i)] $(\SSF , \| \, \cdot \, \|_{\SSF})$ is a Banach space,
	\item[ii)] $\|XYZ\|_{\SSF}\leq \|X\|\|Y\|_{\SSF}\|Z\|$, whenever $X,Z \in \B$ and $Y \in \SSF$,
	\item[iii)] $\|X\|_{\SSF}=\|X\|$, when $X$ has rank one.   
\end{enumerate}

Well-known examples of symmetrically-normed ideals are the $p$-Schatten operators $\SSF_p$ for $1\leq p \leq \infty$, where $\SSF_{\infty}$ stands for the compact operators.  More elaborated examples as well as a full treatment of symmetrically-normed ideals can be found in  
\cite{gohbergk60} or \cite{simon05}.

 \begin{corollary}\label{p norm}
Let $\mathfrak{S}$ be a symmetrically-normed ideal of $\BH$. Then $\CN$ is a real analytic submanifold of $\mathfrak{S}$. 
\end{corollary}
\begin{proof}
We apply again Lemma \ref{Raeburn} with $X=\mathfrak{S}$, $G=\UC$ and $V \in \CN$. Note that 
the action  $\UC \times\mathfrak{S} \to \mathfrak{S}$, $(U,X)\mapsto UX$ is an analytic map since the inclusion map 
$i:\UC \hookrightarrow Gl(\Hb^{1})$ is analytic (see Remark \ref{inclusion smooth}) and the bilinear map 
\[  \BH \times \mathfrak{S} \longrightarrow \mathfrak{S}, \, \, \, \, \, \, (X,Y) \mapsto XY \]
is bounded. In fact, this follows from the very definition of symmetrically-normed ideals, since it is assumed that $\|XY\|_{\mathfrak{S}}\leq \| X\| \|Y\|_\mathfrak{S}$.  On the other hand, we claim that the map
\[  \pi_V : \UC \longrightarrow \mathfrak{S}, \, \, \, \, \, \, \pi_V(U)=UV \]
has continuous local cross sections. To this end let $V_1, V_2 \in \CN$. Since  $V_1 - V_2$ has rank at most $2N$, there are at most $2N$ nonzero singular values $s_j(V_1 - V_2)$, $j=1, \ldots, 2N$, when one regards $V_1 - V_2$ as an operator acting on $\Hb^{1}$.  It follows that   
\[   \| V_1 - V_2 \| \leq \| V_1 - V_2 \|_{\mathfrak{S}}  \leq \sum_{j=1}^{2N} s_j(V_1 - V_2 ) \leq 2N \| V_1 - V_2 \|.  \]
 Hence the local cross sections constructed in Proposition \ref{sections cn} are continuous with respect the norm of $\mathfrak{S}$. 
Finally,  note that tangent spaces may be rewritten as
\[   
(T\CN)_V = \{ \, XV  \,    : \, X \in \uF \cap   \SSF  \, \}.   
\]
This allows us to find closed supplements using the same method as in the previous theorem, but with $\SSF$ in place of $\BH$ in the 
definition of the map $K$. 
\end{proof}

\section{Grassmann manifolds in Quantum Chemistry}
\label{Grassmann}
In the very beginning of the Introduction, the Grassmann manifold $\GN$ in Quantum Chemistry was defined as a quotient space of $\CN$, when the later was considered as a subset of $(\Hb^{1})^N$. If we think of $\CN$ as operators, the  
\textit{Grassmann manifold in Quantum Chemistry} may be defined as the quotient space with respect to the equivalence relation
\[    V_1 \sim V_2 \, \text{ if } \, V_1 U =V_2 \, \text{ for some $U \in \UC (S)$ }, \]
where $S$ is an $N$-dimensional subspace of $\Hb^{1}$ equal to the initial space of the operators in $\CN$ and $\UC (S)$ 
denotes the unitary group of $S$ with respect the $L^2$ inner product. 

\begin{remark}In the above definition of the equivalence relationship we may assume that $U \in \UC$. In fact, if $V_1U=V_2$ for some $U \in \UC$, then $\|V_1U\xi\|_{L^2}=\|V_2\xi\|_{L^2}=\|\xi\|_{L^2}=\|U\xi\|_{L^2}$, and consequently, $U\xi \in S$, whenever $\xi \in S$. We thus get $U|_S \in \UC(S)$, and $V_1 U|_S=V_2$.   
\end{remark}




\noi Let $\PN$ denote the set of rank $N$ $L^2$-orthogonal projections on $\Hb^{1}$, i.e. 
\[  \PN=\{   \, P \in \BH  \, : \, P^2=P, \, \PIL{P\xi}{\eta}=\PIL{\xi}{P\eta}, \, \forall \, \xi, \eta \in \Hb^{1}   \, \}. \]
In the sequel, we regard $\PN$ endowed with the topology inherited from $\BH$.

\begin{remark}
Note that $\PN$ can be characterized as 
\[   \PN=\bigg\{  \, \sum_{i=1}^N \PIL{\, \cdot \, }{\eta_i}\eta_i  \, : \,( \eta_1 , \ldots , \eta_N) \in \CN \, \bigg\}. \]
Actually, if  $P \in \PN$ and $\eta_1 , \ldots \, \eta_N$ is an $L^2$ orthonormal basis of $\Ran(P)$,  then  for any $\xi \in \Hb^{1}$, we have that $\xi=P\xi + (I-P)\xi=\sum_{i=1}^N \PIL{P\xi}{\eta_i}\eta_i + (I-P)\xi$. It follows
$P\xi=\sum_{i=1}^N \PIL{P\xi}{\eta_i} P\eta_i=\sum_{i=1}^N \PIL{\xi}{P\eta_i}\eta_i=\sum_{i=1}^N \PIL{\xi}{\eta_i}\eta_i$. The other inclusion is trivial. 
\end{remark}

\begin{lemma}\label{grass}
The map $$\varphi: \CN \longrightarrow \PN, \, \, \, \, \, \varphi(V)=VV^{\ast},  $$
has continuous local cross sections. In particular, $\GN$ and $\PN$ are homeomorphic.
\end{lemma}
\begin{proof}
Let $P \in \PN$.  Consider the  open neighborhood of $P$ given by
\[   \mathcal{V}= \bigg\{ \, P_1 \in \PN \, : \, \| P- P_1 \| < \frac{1}{(\|P\|  + 1)^2}  \, \bigg\}.  \]
Then we note that for $P_1 \in \mathcal{V}$, 
\[  \| P - PP_1P \| = \| P (P -P_1)P \| \leq \|P \|^{2} \|P- P_1 \| < 1, \]
and 
\[  \| P_1 - P_1PP_1 \| \leq \|P_1 \|^{2} \|P- P_1 \|  \leq (\| P \| + 1)^2 \| P - P_1 \| < 1. \]
In the same fashion as the proof of Proposition \ref{sections cn} we can construct an $L^2$-partial isometry $T_1=T_1(P_1)=P_1(PP_1P)^{-1/2}$ such that $T_1^{\ast}T_1=P$ and $T_1T_1^{\ast}=P_1$. In order to modify the initial space of $T_1$, we can proceed as in the proof of Lemma \ref{transitive} to find an operator $U \in \UC$ such that $U(S)=P(\Hb^{1})$, where $S$ is the initial space of operators in $\CN$. Thus the continuous map $\psi: \mathcal{V} \longrightarrow \CN$, $\psi(P_1)=T_1 (P_1)U$, is the required section for $\varphi$.
 
 Next we consider the following commutative diagram
\begin{displaymath}
\xymatrix{
\CN \ar[r]^{\tilde{\varphi}} \ar[dr]_{\varphi} & \GN \ar[d]^{i}    \\
          & \PN  ,          }
\end{displaymath}
where $\tilde{\varphi}(V)=[V]$ and $i([V])=VV^{\ast}$. Notice that  $\varphi(V)=VV^{\ast}=V_1V_1^{\ast}= \varphi(V_1)$ implies $V=VV^{\ast}V=V_1V_1 ^{\ast} V$. Then we have that 
\[  U=\left(\begin{array}{cc}{V_1 ^{\ast} V}&{0}\\{0}&{I}\end{array}\right) \] 
satisfies $V=V_1 U$. So we obtain $[V]=[V_1]$. Therefore $i$ is a bijection. Moreover, $i$ is continuous: let $\mathcal{W}$ an open set in $\PN$, then $i^{-1}(\mathcal{W})$ is open if and only if $\varphi^{-1}(\mathcal{W})$ is open, which follows from the continuity of $\varphi$. Finally, the fact that $i^{-1}$ is continuous is a consequence of the existence of  continuous local cross sections for $\varphi$.  
\end{proof}

\begin{notationnf}By the above result, we will use the symbol $\GN$ to indicate any of the possible presentations of the Grassmann manifold in Quantum Chemistry, i.e.  as a quotient space or rank $N$ $L^2$-orthogonal projections.  
\end{notationnf}

\medskip

\subsection{Differential structure of $\GN$}

In this section we use the previous results on $\CN$ to study   the differential structure of $\GN$.  First we  define  an action of the Banach-Lie group $\UC$  on $\GN$  by
$$\UC \times \GN \longrightarrow \GN, \, \, \, \, \, \, \,  U \cdot P=UPU^{-1}.$$ 

\begin{remark}\label{action trans p}
Note that this action is transitive: let $P,P_1 \in \GN$. By Lemma \ref{grass} there are $V,V_1  \in \CN$ such that $VV^{\ast}=P$ and $V_1V_1^{\ast}=P_1$. Applying  Lemma \ref{transitive} we get $U \in \UC$ such that $UV=V_1$. Let $W \in \UC(L^2)$ be the extension of $U$ to all $L^2$.   Then we have that $WPW^{\ast}=WV(WV)^{\ast}=V_1 V_1 ^{\ast}=P_1$ on $L^2$, but this yields $UPU^{-1}=P_1$ when one restricts the operators to $\Hb^{1}$.                                                     
\end{remark}

\begin{lemma}\label{section gn}
Let $P \in \GN$. The map
\[  \pi_P : \UC \longrightarrow \GN, \, \, \, \, \, \, \, \pi_P(U)=UPU^{-1}, \]
has continuous local cross sections. In particular, it is locally trivial fiber bundle.  
\end{lemma}
\begin{proof}
Let $P,P_1 \in \GN$. Set $V=\psi (P)$, where $\psi$ is the continuous local cross section in the proof of Lemma \ref{grass}. According to Proposition \ref{sections cn}  there exists an open neighborhood $\mathcal{W}$ of $V$ and a continuous map $\sigma: \mathcal{W} \longrightarrow \UC$ such that $\sigma(V_1)V=V_1$ for all $V_1 \in \mathcal{W}$. Then  we choose $r >0$ to ensure that $ \psi(P_1) \in \mathcal{W}$ whenever $\|P_1 - P\|<r$, and we set
\[  \phi: \{  \, P_1  \in \GN\,: \, \| P_1 - P\| < r \,  \} \longrightarrow \UC, \, \, \, \, \, \, \phi(P_1)=(\sigma \circ \psi )(P_1). \]
Clearly,  $\phi$  is a continuous map. Note that we can extend the operators of the range of $\sigma$ to obtain unitary operators on $L^2$  satisfying
\[  \pi_P(\phi(P_1))=\phi(P_1)P\phi(P_1)^{\ast}=\sigma(\psi(P_1))V (\sigma(\psi(P_1))V)^{\ast}=\psi(P_1)\psi(P_1)^{\ast}=P_1\,.  \]
If we restrict the above equation to $\Hb^{1}$, we find that $\phi$ is a section for $\pi_P$, and this ends the proof.      
\end{proof}

\begin{remark}
Let $P \in \GN$. The isotropy group of the action of $\UC$ on $\GN$ is given by
\[ G_P=\{  \,   U \in \UC \, : \,  U P=PU    \,  \}.   \]
Operators in $G_P$  can be regarded as block diagonal operators with respect the projection $P$, i.e. 
\[    U=\left(\begin{array}{cc}{U_{11}}&{0}\\{0}&{U_{22}}\end{array}\right).  \]
 Then the Lie algebra of $G_P$ is given
\[ \mathfrak{g}_P = \{  \,  X \in \uF    \, : \, XP=PX  \, \},  \]
or in terms of block matrices any $X \in \mathfrak{g}_P$ is of the form
\[    X=\left(\begin{array}{cc}{X_{11}}&{0}\\{0}&{X_{22}}\end{array}\right).  \]  
\end{remark}

\begin{theorem}
Let $P \in \GN$. Then the map 
\[ \pi_P : \UC \longrightarrow \GN, \, \, \, \, \, \, \pi_P(U)=UPU^{-1}\]
is a real analytic submersion, and induces on $\GN$ a homogeneous structure. Furthermore, $\GN$ is a real analytic submanifold of $\BH$. 
\end{theorem}
\begin{proof}
Note that the action of $\UC$ on $\BH$ given by $(U,X)\mapsto UXU^{-1}$ is analytic essentially due to the fact  that the inclusion map $\UC \hookrightarrow \GH$ is analytic, which was pointed out in Remark \ref{inclusion smooth}. According to Lemma \ref{section gn}  the map $\pi_P: \UC \longrightarrow \GN$, $\pi_P(U)=UPU^{-1}$  is open. Its differential at the identity is given by
\[   \delta_P:=(d\pi_P)_I: \uF \longrightarrow \BH, \, \, \, \, \, \, \, \delta_P(X)=XP-PX. \] 
Note that  kernel of $\delta_P$ is a closed complemented subspace of $\uF$. In fact, a closed supplement is given by the co-diagonal block matrices, i.e.
\[  \mathfrak{h}_P= \{ \,  X \in \uF \, : \, PXP=(I-P)X(I-P)=0 \,    \}.   \] 
On the other hand, the range of $\delta_P$ is also a closed complemented subspace of $\BH$. To this end we remark that the argument given in 
\cite{AL08} still works in our setting. Indeed, it can be showed that $\delta_P \circ \delta_P \circ \delta_P =\delta_P$, and hence 
$E:=\delta_P \circ \delta_P$ is a continuous idempotent onto the range of $\delta_P$. Then we can apply Lemma  \ref{Raeburn}, and 
the proof is complete.  
\end{proof}

\noi The following result can be drawn in much the same way as  Corollary \ref{p norm}. 

\begin{corollary}\label{p norm grassmannian}
Let $\mathfrak{S}$ be a symmetrically-normed ideal of $\BH$. Then $\GN$ is a real analytic submanifold of $\mathfrak{S}$. 
\end{corollary}

\section{Finsler structures for the Stiefel and Grassmann manifolds in Quantum Chemistry}
\label{Finsler-Riemann}
As a straightforward application of the preceding results, we show that the Stiefel and Grassmann manifolds in Quantum Chemistry are complete Finsler manifolds. Furthermore, there is a natural Riemannian metric for these manifolds induced by the Hilbert-Schmidt inner product.  The motivation for including these consequences is that Finsler and Riemannian manifolds provide a quite natural framework in critical point theory (see e.g. \cite{deimling85}, \cite{gh93}).

 Since the notion of Finsler manifolds is not uniform in the literature, we  mention that we use the definition  introduced by  Palais 
 \cite{palais66}. Let $M$ be  a $C^1$  manifold  modeled on a Banach space $\Xc$ with tangent bundle $TM$. A  \textit{Finsler structure} for $M$ is a  function $ \| \, \cdot \| : TM \longrightarrow \RR$ such that 
 \begin{itemize}
	\item[i)] for each $m \in M$, $w \in (TM)_m$, the function $(x,w) \mapsto \|w\|_m:=\| (m,w) \|$ is an admissible norm on $(TM)_m$,
	\item[ii)] for each $m_0 \in M$, $(\mathcal{W},\Phi)$ a chart of $M$ with $m_0 \in \mathcal{W}$ and $k>1$, there is an open neighborhood  $\mathcal{W}_{m_0}$ of $m_0$ in $\mathcal{W}$ satisfying
\[    
\frac{1}{k}\| d\Phi^{-1}_{\Phi(m)}  (v)\|_m  \leq  \| d\Phi^{-1}_{\Phi(m_0)} (v)\|_{m_0} \leq  k\| d\Phi^{-1}_{\Phi(m)} (v)\|_m 
\]    
for all $m \in \mathcal{W}_{m_0}$ and all $v \in \Xc$. 
\end{itemize}
  A \textit{Finsler manifold} is a $C^1$ Banach manifold together with a Finsler structure. If $\gamma(t)$, $t \in [0,1]$, is a $C^1$ curve in $M$, its length is defined by
\[  L(\gamma)=\int_0 ^1 \| \dot{\gamma}(t)\|_{\gamma(t)} \, dt. \]
On each connected component of $M$, there is a well defined metric given by
\[   d(m_0,m_1)= \inf \{ \, L(\gamma)  \,  :   \gamma \subseteq M , \,  \gamma(0)=m_0 ,\, \gamma(1)=m_1    \,  \}, \]
where the curves considered are $C^1$. 
Furthermore, it turns out that the topology defined by this metric $d$ coincides with the  manifold topology  of $M$.   We refer the reader to  \cite{deimling85} or \cite{palais66} for 
the proof of these facts. If $(M,d)$ is a complete metric space, then $M$ is called a \textit{complete Finsler manifold}.

\begin{remark}
\label{complete Finsler}
An example of a complete Finsler manifold is a closed $C^1$ submanifold $M$ of a Banach space $\Xc$ endowed with the norm induced by the inclusion $(TM)_m \subseteq (T\Xc)_m \simeq \Xc$ (\cite[Theorem 3.6]{palais66}).
\end{remark} 
 
Let $V \in \CN$. Recall that the map $\pi_V$ is a submersion and, therefore, the tangent space of $\CN$ at $V$ may be identified with
\[  (T\CN)_{V}=\{ \, XV \, : \, X \in \uF \, \}. \] 
Let $P \in \GN$. Analogously,  we may identify the tangent space of $\GN$ at $P$ with
\[   (T\GN)_{P}=\{ \,  XP-PX  \, :  \, X \in \uF \,   \}. \]  

\begin{corollary}
Let $\SSF$ be a symmetrically-normed ideal of $\BH$. The following assertions hold:
\begin{itemize}
	\item[i)] $\CN$ is a complete Finsler manifold with  structure given by
	\[  \| XV \|_V :=  \|  XV \|_{\SSF}, \, \, \, \, \, \, \, \, XV \in (T\CN)_V.  \]
	\item[ii)] $\GN$ is a complete Finsler manifold with structure given by
	\[  \| XP -PX \|_P :=  \|  XP-PX \|_{\SSF}, \, \, \, \, \, \, \, \, XP-PX \in (T\GN)_P.  \]
\end{itemize}
\end{corollary} 
\begin{proof}
It follows from  the fact that both manifolds are closed in $\SSF$, Corollary \ref{p norm grassmannian}, Corollary \ref{p norm} and    Remark \ref{complete Finsler}. 
\end{proof}

As a special case of the above corollary, a bit more can be stated when one considers the ideal of Hilbert-Schmidt operators of $\BH$. In fact, a Riemannian metric on $\CN$ is at hand: for  $XV, YV \in (T\CN)_V$,
\[ \PI{XV}{YV}_V := \re \, \Tr( XV (YV)^{\ast}), 
\]
where $\Tr$ is the usual trace and the  adjoint  is taken with respect to the $\Hb^{1}$ inner product. In a similar fashion, we can define a  Riemannian metric on the Grassmann manifold:  given $XP-PY, YP-PY \in (T\GN)_P$,
\[ 
\PI{XP- PX}{YP-PY}_P := \re \, \Tr( (XP-PX) (YP-PY)^{\ast}) .  
\]

\begin{corollary}
$\CN$ and $\GN$ are   complete analytic Hilbert-Riemann manifolds. 
\end{corollary} 

\begin{remark}
Let $\mathbb{S}^{M}$ denote the unit sphere in $\RR^{M+1}$.  In the multi-configurative Hartree-Fock type equations \cite{lewin04}, the energy functional is now defined in the variational spaces  
$  \mathcal{C}_{K,N}:= \mathbb{S}^{\binom{K}{N}} \times \mathcal{C}_K$, where  $N<K$. Thus $\mathcal{C}_{K,N}$ is also a  complete analytic Hilbert-Riemann manifold, being the product of a sphere and  the Stiefel manifold $\mathcal{C}_K$.
\end{remark}

\end{document}